\documentclass{article}
\usepackage[utf8]{inputenc}
\usepackage{listings}
\usepackage[numbers]{natbib}
\usepackage{amsmath}
\usepackage{amsthm}
\usepackage{amsfonts}
\usepackage{amssymb}
\usepackage[utf8]{inputenc}
\usepackage[english]{babel}
\usepackage[a4paper]{geometry}
\usepackage{graphicx}
\usepackage{color,psfrag}
\usepackage{xcolor}
\usepackage{fancyhdr}
\usepackage{makeidx}
\usepackage{enumerate}
\usepackage{mathpazo}
\usepackage[boxed]{algorithm2e}
\usepackage{mathrsfs}
\usepackage{float}
\usepackage{datetime}
\usepackage{titlesec}
\usepackage{enumerate}
\usepackage{enumitem}
\usepackage[labelfont=bf,textfont=it,labelsep=period,width=.8\textwidth]{caption}
\usepackage[unicode,pdfstartview={XYZ null null 1},pdfview={XYZ null null 1},pdfpagemode=UseOutlines]{hyperref}
\usepackage[all]{xy}
\usepackage{comment}

\DeclareFontFamily{U}{wncy}{}
\DeclareFontShape{U}{wncy}{m}{n}{<->wncyr10}{}
\DeclareSymbolFont{mcy}{U}{wncy}{m}{n}
\DeclareMathSymbol{\Sh}{\mathord}{mcy}{"58}

\def\tr{{ tr}}

\newdate{date}{10}{03}{2017}

\newcommand{\Fq}{\mathbb F_q}

\newcommand{\Cor}{\mathrm{Cor}}
\newcommand{\Disc}{\mathrm{Disc}}

\newtheorem{theorem}{Theorem}[section]

\newtheorem{lemma}[theorem]{Lemma}

\newtheorem{corollary}[theorem]{Corollary}
\newtheorem{proposition}[theorem]{Proposition}

\newtheorem{definition}[theorem]{Definition}

\newtheorem{algor}[theorem]{Algorithm}
\newtheorem{problem}{Problem}
\theoremstyle{remark}
\theoremstyle{remark}\newtheorem{remark}[theorem]{Remark}
\title{Explicit isomorphisms of quaternion algebras over quadratic global fields}
\author{Tímea Csahók\footnote{University of Oxford}, Péter Kutas\footnote{Eötvös Loránd University and University of Birmingham}, Mickaël Montessinos\footnote{Vilnius University, Faculty of Mathematics and Informatics, Institute of Mathematics},Gergely Zábrádi\footnote{Eötvös Loránd University and Rényi Institute of Mathematics, Lendület ``Automorphic'' Research Group}}
\begin{document}

\maketitle
\begin{abstract}
\noindent Let $L$ be a separable quadratic extension of either $\mathbb{Q}$ or $\mathbb{F}_q(t)$. We propose efficient algorithms for finding isomorphisms between quaternion algebras over $L$. Our techniques are based on computing maximal one-sided ideals of the corestriction of a central simple $L$-algebra.  
\end{abstract}

\section{Introduction}

In this paper we consider a special case of the following algorithmic problem. Let $K$ be a global field and let $A$ and $B$ be central simple algebras over $K$ given by a $K$-basis and a multiplication table of the basis elements. The coefficients in the multiplication table are called structure constants. The task is to decide whether $A$ and $B$ are isomorphic, and if so, find an explicit isomorphism between them. A special case of this problem when $B=M_n(K)$ is referred to as the \textit{explicit isomorphism problem} which has various applications in arithmetic geometry \cite{cremona2015explicit},\cite{fisher2013explicit},\cite{fisher2014computing}, computational algebraic geometry \cite{de2006lie} and coding theory \cite{gomez2016new},\cite{gomez2020primitive}. 

\sloppy In 2012, Ivanyos, Rónyai and Schicho \cite{ivanyos2012splitting} exhibited an algorithm for the explicit isomorphism problem in the case where $K$ is an algebraic number field. Their algorithm is a polynomial-time ff-algorithm (which means one is allowed to call an oracle for factoring integers and polynomials over finite fields) in the case where the dimension of the matrix algebra, the degree of the number field and the discriminant of the number field are all bounded. More concretely, the running time of the algorithm is exponential in all these parameters. They also show that finding explicit isomorphisms between central simple $K$-algebras of dimension $n^2$ over $K$ can be reduced to finding an explicit isomorphism between an algebra $A$ and $M_{n^2}(K)$. 

Then in \cite{fisher2017higher} and independently in \cite{kutas2019splitting} an algorithm was provided when $A$ is isomorphic to $M_2(\mathbb{Q}(\sqrt{d}))$ where the algorithm is polynomial in $\log(d)$. 
The case where $K=\mathbb{F}_q(t)$, the field of rational functions over a finite field was considered in \cite{ivanyos2018computing} where the authors propose a randomized polynomial-time algorithm. The algorithm is somewhat analogous to the algorithm of \cite{ivanyos2012splitting} but it is polynomial in the dimension of the matrix algebra. Similarly to the number field case, this was extended to quadratic extensions (now with a restriction to odd characteristics) in \cite{ivanyos2019explicit}. 


In this paper we initiate a new method for dealing with field extensions which is analogous to Galois descent. It is known that finding an explicit isomorphism between $A$ and $M_n(K)$ is polynomial-time equivalent to finding a rank 1 element in $A$. Thus if one could find a subalgebra of $A$ isomorphic to $M_n(\mathbb{Q})$ or $M_n(\mathbb{F}_q(t))$, then one could apply the known algorithms for the subalgebra and that would give an exponential speed-up in both cases. Furthermore, these types of methods should work equally for the function field and number field case which have completely different applications. In \cite{kutas2019splitting} and \cite{ivanyos2019explicit} this type of method is studied. In both cases one finds a central simple algebra over the smaller field in $A$ which is not necessarily a matrix algebra but when it is a division algebra, then it is split by the quadratic field (the center of $A$) which can be exploited. The disadvantage of these methods is that they are based on explicit calculations and reductions to finding nontrivial zeros of quadratic forms which do not generalize easily to higher extensions.

In this paper we reprove results of \cite{kutas2019splitting} in a more conceptual way and extend them to the isomorphism problem of two quaternion algebras over a quadratic extension. The main technique is to compute a maximal right ideal of the corestriction of the algebra $A$ (which is an explicit construction corresponding to the usual corestriction on cohomology groups) and apply it to construct an involution of the second kind on $A$. In general this might not be useful, but when $A$ possesses a canonical involution of the first kind, then composing the two kinds of involutions and taking fixed points gives us the central simple subalgebra over a smaller field. Fortunately, tensor products of quaternion algebras carry a canonical involution of the first kind which is exactly what we need. This provides an example of the explicit isomorphism problem when the degree of the field over $\mathbb{Q}$ or $\mathbb{F}_q(t)$ is fixed but the discriminant does not need to be bounded. 

We also implement our algorithm in Magma \cite{bosma1997magma}. In particular, this also involved implementing the main algorithm from \cite{ivanyos2018computing} and \cite{gomez2020primitive}. The same implementation was used in \cite{csahok2022finding} for matrix algebras of degree $2$ in even characteristic. Here we use it for algebras of higher degree and study its efficiency. Even though our main algorithm runs in polynomial time, the implementation is not practical. The bottleneck of the computation seems to be computing maximal orders in higher degree split central simple algebras. The computationally expensive part is not the factorization of the discriminant of the starting order (which in the rational function field case is particularly fast), just the fact that the currently known maximal order algorithms run in polynomial time but with a large exponent. We analyze the complexity of maximal order algorithms given in \cite[section 3]{ivanyos2018computing} and \cite[sections 3 and 4]{friedl1985polynomial} and we also provide some substantial speed-ups in the case relevant to our main algorithm (when the algebra is obtained as a corestriction). 

The paper is structured as follows. Section 2 contains number theoretic and algorithmic preliminaries. Section 3 is devoted to the general method of computing involutions of the second kind and computing Galois descents of quaternion algebras. In Section 4 we describe our main algorithm for finding explicit isomorphisms between quaternion algebras over quadratic extensions of either $\mathbb{Q}$ or $\mathbb{F}_q(t)$ (where $q$ can be even as well). Section 5 is devoted to complexity estimates and optimisation tricks to speed-up the computations. Section 6 contains some details about our Magma implementation\footnote{\url{https://github.com/QuaternionIsomorphisms/QuaternionIsomorphisms/}}.

\subsection*{Acknowledgements}
We would like to thank John Voight for helpful suggestions and comments on an earlier version of this manuscript.

\section{Preliminaries}
\subsection{General algebraic background}

\begin{definition}
Let $K$ be a field and let $A$ be a finite dimensional algebra over $K$. Then $A$ is a \textbf{central simple algebra} over $K$ if it is simple and its center $Z(A)$ equals $K$ (central). A central simple algebra $A$ over the field $K$ that has dimension 4 over $K$ is called a \textbf{quaternion algebra}.
\end{definition}

By a fundamental result of Wedderburn, a central simple algebra $A$ is isomorphic to the full  matrix algebra $M_n(D)$ for some division ring $D$. In particular, a quaternion algebra over $K$ is either a division algebra or is isomorphic to the algebra of $2\times 2$ matrices over $K$. 

\begin{definition}
Let $A$ be a central simple algebra over $K$. We say that $A$ is \textbf{split} by a field extension $L/K$ if $A\otimes_K L\simeq M_n(L)$ for a sufficient $n$. If a central simple algebra over $K$ is isomorphic to $M_n(K)$, then we call the algebra split (i.e. a shorter version of split by the extension $K/K$). 
\end{definition}

Now we recall some facts about the Brauer group. Our main reference is \cite{gille2017central}.

\begin{definition}
We call the central simple $K$-algebras $A$ and $B$ \textbf{Brauer equivalent} if there exist integers $m,m'>0$ such that $A\otimes_K M_m(K)\cong B\otimes_K M_{m'}(K)$. The Brauer equivalence classes of central simple $K$-algebras form a group under tensor product over $K$. This group is called the \textbf{Brauer group} $\operatorname{Br}(K)$ of $K$.
\end{definition}

In order to state the cohomological interpretation of the Brauer group we need to indroduce some further notation. For a field $K$ we put $K_{sep}$ for a fixed separable closure of $K$ and $G_K:=\operatorname{Gal}(K_{sep}/K)$ for the absolute Galois group. 

\begin{theorem}\cite[Thm.\ 4.4.3]{gille2017central} 
Let $K$ be a field. Then the Brauer group $\operatorname{Br}(K)$ is naturally isomorphic to the second Galois cohomology group $ H^2(G_K,K_{sep}^\times)$.
\end{theorem}

For specific fields one can even determine the Brauer group explicitly. The case of local fields is treated by the following famous result of Hasse.

\begin{proposition}[Hasse] \cite[Prop.\ 6.3.7]{gille2017central}\label{prop:Hasse}
Let $K$ be a complete discretely valued field with finite residue field. Then we have a canonical isomorphism $$\operatorname{Br}(K)\cong\mathbb{Q}/\mathbb{Z}\ .$$
Moreover for a finite separable extension $L/K$ there are commutative diagrams
\begin{align*}
\xymatrix{
\operatorname{Br}(L)\ar[r]^\cong \ar[d]_{\operatorname{Cor}} & \mathbb{Q}/\mathbb{Z}\ar[d]^{\operatorname{id}}\\
\operatorname{Br}(K)\ar[r]^\cong & \mathbb{Q}/\mathbb{Z}
}\quad\text{ and }\quad
\xymatrix{
\operatorname{Br}(K)\ar[r]^\cong \ar[d]_{\operatorname{Res}} & \mathbb{Q}/\mathbb{Z}\ar[d]^{\operatorname{|L\colon K|}}\\
\operatorname{Br}(L)\ar[r]^\cong & \mathbb{Q}/\mathbb{Z}\ ,
}
\end{align*}
where the right vertical map in the second diagram is the multiplication by the degree $|L\colon K|$.
\end{proposition}

The map inducing the isomorphism $\operatorname{Br}(K)\cong \mathbb{Q}/\mathbb{Z}$ is classically called the \textbf{Hasse invariant map}. Note that in the archimedean case Frobenius' Theorem on division rings over the real numbers $\mathbb{R}$ is equivalent to the fact $\operatorname{Br}(\mathbb{R})=\frac{1}{2}\mathbb{Z}/\mathbb{Z}\subset \mathbb{Q}/\mathbb{Z}$. Finally, since $\mathbb{C}$ is algebraically closed, we have $\operatorname{Br}(\mathbb{C})=0$.

Now let $K$ be a \emph{global field}, i.e\ either a number field (finite extension of $\mathbb{Q}$) or the function field $K=\mathbb{F}(C)$ of a smooth projective curve $C$ over a finite field $\mathbb{F}$. Denote by $\mathcal{P}$ the set of (finite and infinite) places of $K$, ie.\ in the function field case $\mathcal{P}$ is the set $C_0$ of closed points on $C$ and in the number field case $\mathcal{P}$ consists of the prime ideals in the ring of integers of $K$ and the set of equivalence classes of archimedean valuations on $K$. For a place $P\in\mathcal{P}$ we denote by $K_P$ the completion of $K$ at $P$. If $A$ is a central simple algebra over $K$ then $A_P:=A\otimes_K K_P$ is a central simple algebra over $K_P$. This induces a natural map $\operatorname{Br}(K)\to \operatorname{Br}(K_P)\overset{\operatorname{inv}_P}{\to} \mathbb{Q}/\mathbb{Z}$. Note that every central simple algebra $A$ splits at all but finitely many places, ie.\ we have $\operatorname{inv}_P([A_P])=0$ for all but finitely many $P$. Using the main results of class field theory one obtains the following classical theorem of Hasse.

\begin{theorem}[Hasse] \cite[Cor.\ 6.5.3, Rem.\ 6.5.5]{gille2017central} \label{localglobalbrauer} For any global field $K$ we have an exact sequence
$$
0\to \operatorname{Br}(K)\to \bigoplus_{P\in \mathcal{P}}\operatorname{Br}(K_P)\overset{\sum\operatorname{inv}_P}{\to}\mathbb{Q}/\mathbb{Z}\to 0\ .
$$
\end{theorem}
Note that the Hasse-invariant of a nonsplit quaternion algebra over a local field is $\frac{1}{2}$. In particular, any quaternion algebra $A$ over $K$ splits at an even number of places. Further, for any finite subset $S\subset \mathcal{P}$ of even cardinality there exists a unique quaternion algebra (upto isomorphism) over $K$ that splits exactly at the places in $\mathcal{P}\setminus S$. This is usually referred to as Hilbert's reciprocity law.

Finally, we briefly recall the definition and basic properties of orders in central simple algebras over local and global fields.
\begin{definition}
Let $R$ be a Dedekind domain and $K$ be its field of fractions. An $R$-order in a central simple algebra $A$ over $K$ is a subring $O$ in $A$ that is a finitely generated $R$-submodule in $A$ such that $K\cdot O=A$ (ie.\ $O$ is a full $R$-lattice in the $K$-vectorspace $A$). We call an order $O\subset A$ maximal if it is maximal with respect to inclusion.
\end{definition}

By the following result, being a maximal order is a local property.

\begin{theorem}\cite[Cor.\ 11.2]{reiner2003book}
An $R$-order $O$ in $A$ is maximal if and only if for each maximal ideal $P$ in $R$ the localization $O_P$ is a maximal $R_P$-order in $A$.
\end{theorem}

\subsection{The corestriction of a central simple algebra}\label{corestrictionsec}

Due to the fact that the Brauer group admits a cohomological interpretation, one can use standard techniques from Galois cohomology to analyze central simple algebras. Let $L$ be a finite Galois extension of $K$ (contained in the fixed separable closure $K_{sep}$). Let $G_K$ and $G_L$ be the absolute Galois group of $K$ and $L$ respectively. There are two standard maps to analyze: restriction, which is a map from $H^2(G_K,K_{sep}^\times)$ to $H^2(G_L,K_{sep}^\times)$ and corestriction which is a map from $H^2(G_L,K_{sep}^\times)$ to $H^2(G_K,K_{sep}^\times)$. 

For our purposes we need explicit descriptions of these maps on central simple algebras. The restriction map is easy, one just considers the extensions of scalars by $L$ (ie.\ the map $A\mapsto A\otimes_K L$). However the corestriction map is more complicated. We describe the corestriction map when $L$ is a separable quadratic extension of $K$ (which implies that it is a Galois extension). This discussion is taken from \cite[Section 3B]{knus1998book} (in that book the corestriction is called the norm of an algebra). 

Let $L$ be a separable quadratic extensions of a field $K$. Let $\sigma$ be a generator of $Gal(L/K)$. Let $A$ be a central simple algebra over $L$. Then we define $A^{\sigma}$ to be the algebra where each structure constant of $A$ is conjugated by $\sigma$. Alternatively, one can define $A^{\sigma}$ as a collection of elements $\{a^\sigma|~a\in A\}$ with the following properties:
$$a^{\sigma}+b^{\sigma}=(a+b)^{\sigma},~a^{\sigma}b^{\sigma}=(ab)^{\sigma}~, (\lambda\cdot a)^{\sigma}=\sigma(\lambda)a^{\sigma} .$$
$A^{\sigma}$ is also a central simple $L$-algebra and the induced map $\sigma$ provides a $K$-isomorphism between $A$ and $A^{\sigma}$. 

\begin{definition}\label{def:switch}
Let $L$ be a separable quadratic extension of $K$. Let $A$ be a central simple $L$-algebra. The switch map $s$ is the $K$-linear endomorphism of $A \otimes_L A^{\sigma}$ defined on elementary tensors by $s(a \otimes b^{\sigma}) = b \otimes a^{\sigma}$, extended K-linearly.

\end{definition}
\begin{proposition}\label{prop:cor}\cite[Proposition 3.13.]{knus1998book}
The elements of $A\otimes A^{\sigma}$ invariant under the switch map form a subalgebra which is a central simple algebra over $K$ of dimension $\dim_L(A)^2$ over $K$.
\end{proposition}
The algebra in Proposition \ref{prop:cor} is called the \textbf{corestriction} of $A$ (with respect to the extension $L/K$). It corresponds to the corestriction map of Galois cohomology (and it is also true that $Cor\circ Res$ is multiplication by $n$ in the Brauer group of $K$ but we will not use this fact in this paper). Our main application of the corestriction maps concerns involutions of central simple algebras. Recall that an involution of the central simple algebra $A$ of the \emph{second kind} is an involution whose restriction to the center $L$ of $A$ is nontrivial. For an overview of involutions the reader is referred to \cite[Chapter 1, Section 1-3]{knus1998book}. The main result we use is the following:
\begin{theorem} \label{involutioncorestrictionsplit}
Let $L/K$ be a quadratic Galois extension and let $A$ be a central simple algebra over $L$. Then $A$ admits an involution of the second kind if and only if the corestriction of $A$ is split. 
\end{theorem}
The proof of this theorem in \cite{knus1998book} is constructive which we will exploit in later sections. 

\subsection{Corestriction of maximal orders}

For the purpose of optimising maximal order computation in the corestriction of a matrix algebra (see section \ref{sect:complexity} for details), we need to consider the corestriction construction over Galois extensions of rings. For the convenience of the reader, we recall the key points of this construction for our setting. The discussion is taken from \cite{ford2017book}.

The conceptual definition of a Galois extension of rings requires more machinery than is necessary for our purpose, so we quote as a definition the characterisation given by point (6) of \cite[Theorem 12.2.9]{ford2017book}:
\begin{definition}
Let $R$ be a commutative ring, and $S$ a commutative $R$-algebra. Let $G$ be a finite group of $R$-algebra automorphisms of $S$. Then $S$ is a Galois extension of $R$ with group $G$ if the following conditions are verified:
\begin{enumerate}
    \item $S^G = R$
    \item for each maximal ideal $\mathfrak{m}$ of $S$ and for each non trivial $\sigma \in G$, there is an $x \in S$ such that $\sigma(x)-x \notin \mathfrak{m}$.
\end{enumerate}
\end{definition}

\begin{proposition}\label{prop:quadgalois}
    Let $K$ be a global field and let $L$ be a separable quadratic extension of $K$ with Galois group $G = \{1,\sigma\}$. Let $R \subsetneq K$ be a Dedekind domain, and let $S$ be the integral closure of $R$ in $L$. Then, $S$ is a Galois extension of $R$ with group $G$ if and only if no prime ideal of $R$ is ramified in $L$.
\end{proposition}

\begin{proof}
Since $R = S \cap K$ it is clear that $R = S^G$. Now, we let $\mathfrak{P}$ be a prime ideal of $S$, lying above a prime $\mathfrak{p}$ in $R$. Then if $\mathfrak{p}$ does not ramify in $L$, either $\sigma(\mathfrak{P}) \neq \mathfrak{P}$ or $\sigma$ induces a non-trivial automorphism of the residue field of $\mathfrak{P}$. In both cases, we may find some $x \in S$ such that $\sigma(x) - x \notin \mathfrak{P}$. On the contrary, if $\mathfrak{p}$ ramifies in $S$, $\sigma(\mathfrak{P}) = \mathfrak{P}$ and $\sigma$ acts trivially on the residue field of $\mathfrak{P}$, so for all $x \in S$, $\sigma(x) - x \in \mathfrak{P}$.
\end{proof}

For the rest of the subsection, we keep the notations of proposition \ref{prop:quadgalois}. Let $A$ be a central simple algebra over $L$ and let $\mathcal{O}$ be a $S$-order in $A$. Then we call corestriction of $\mathcal{O}$ the intersection of $\mathcal{O} \otimes_S \mathcal{O}^\sigma$ and the corestriction of $A$. This construction corresponds to the more general construction given in \cite[subsection 14.1.3]{ford2017book} for a module over a Galois ring extension.

In order to show that in the Galois case, the corestriction of a maximal order in a matrix algebra is maximal order of the corestriction, we use theorems related to Azumaya algebras.
To simplify the exposition, we again use as a definition what is given in \cite{ford2017book} as a characterization. Combining theorem 7.1.4 (3) and corollary 1.1.16 (1), we get:
\begin{definition}
    Let $R$ be a commutative ring $R$. An Azumaya algebra over $R$ is a $R$-algebra $A$ that is finitely generated, projective and faithful as a $R$-module and such that the map $s: A \otimes_R A^{op} \rightarrow End_R(A)$ is an isomorphism, where $s$ is defined by $s(a \otimes b)(x) = axb$ for $a,b,x \in A$.
\end{definition}

We may now state the main result of this subsection:

\begin{proposition}\label{prop:coresunram}
    Let $A = M_n(L)$, and let $\mathcal{O}$ be a maximal $S$-order in $A$. Then the corestriction of $\mathcal{O}$ is a maximal order in the corestriction of $A$.
\end{proposition}

\begin{proof}
    Since $A$ and its corestriction are matrix algebras (respectively over $L$ and $K$), their maximal orders are Azumaya algebras (respectively over $R$ and $S$). This follows from \cite[theorem 11.3.14]{ford2017book}, since the Brauer class of a matrix algebra is trivial in the Brauer group of its base field. Furthermore, any $R$-order that is an Azumaya $R$-algebra is a maximal order in the corestriction of $A$. This is the content of \cite[theorem 11.3.11]{ford2017book}.
    
    Now, the result follows directly from \cite[theorem 14.1.9 (1)]{ford2017book}. Indeed, $S$ is free as a $R$-module so the theorem applies, and it states that the corestriction of an Azumaya $S$-algebra is an Azumaya $R$-algebra.
\end{proof}

The existing theory allows us to describe the intersection of $\mathcal{O} \otimes_S \mathcal{O}^\sigma$ and the corestriction of $A$ in the case that $S$ is an unramified extension of $R$. In the proof of proposition \ref{prop:coresram}, we discuss the situation at ramified primes using an explicit computation.

\subsection{Algorithmic preliminaries}
In this subsection we give a brief overview of known algorithmic results in this context and provide more details of the algorithms specifically used in this paper. 

Let $K$ be a field and let $A$ be an associative algebra given by the following presentation. One is given a $K$-basis $b_1,\dots b_m$ of $A$ and a multiplication table of the basis elements, i.e.\ $b_ib_j$ expressed as a linear combination $\sum_{k=1}^m \gamma_{i,j,k} b_k$. These $\gamma_{i,j,k}$ are called structure constants and we consider our algebra given by structure constants. 
 It is a natural algorithmic problem to compute the structure of $A$, i.e., compute its Jacobson radical $\operatorname{rad} A$, compute the Wedderburn decomposition of $A/\operatorname{rad} A$ and finally compute an explicit isomorphism between the simple components of $A/\operatorname{rad} A$ and $M_n(D_i)$ where the $D_i$ are division algebras over $K$ and $M_n(D_i)$ denotes the algebra of $n\times n$ matrices over $D_i$. The problem has been studied for various fields $K$, including finite fields, the field of complex and real numbers, global function fields and algebraic number fields. There exists a polynomial-time algorithm for computing the radical of $A$ over any computable field \cite{cohen1997finding}. There also exist efficient algorithms for every task over finite fields \cite{friedl1985polynomial},\cite{ronyai1990computing} and the field of real and complex numbers \cite{eberly1991decompositions}. Finally, when $K=\mathbb{F}_q(t)$, the field of rational functions over a finite field $\mathbb{F}_q$, then there exist efficient algorithms for computing Wedderburn decompositions \cite{ivanyos1994decomposition}. 
 
 This motivates the algorithmic study of computing isomorphisms between simple algebras. Over finite fields every simple algebra is a full matrix algebra. Finding isomorphisms between full matrix algebras can be accomplished in polynomial time using the results from \cite{friedl1985polynomial} and \cite{ronyai1990computing}. Now we turn our attention to global fields. 
 
 \subsection{Number fields}
 
 Over number fields there is an immediate obstacle. Rónyai \cite{ronyai1987simple} showed that this task is at least as hard as factoring integers. However, in most interesting applications factoring is feasible, thus it is a natural question to ask whether such an isomorphism can be computed if one is allowed to call an oracle for factoring integers. In \cite{ivanyos2012splitting} the authors propose such an algorithm when $A\cong M_n(K)$ where $K$ is a number field. We sketch the steps of the algorithm here in the $K=\mathbb{Q}$ case: 
 \begin{enumerate}
     \item Compute a maximal order $O$ in $A$
     \item Embed $A$ into $M_n(\mathbb{R})$ to obtain a norm on $A$
     \item Find a reduced basis $b_1,\dots,b_{n^2}$ of $O$
     \item Search through all the elements of small norm and check whether they are of rank 1
     \item A rank one element generates a minimal left ideal, the action of $A$ on the minimal left ideal provides an explicit isomorphism between $A$ and $M_n(\mathbb{Q})$
\end{enumerate}
\begin{remark}
If one of the $b_i$ is a zero divisor (which should happen very rarely), then one can reduce the entire problem to a smaller $n$ and restart the algorithm. It can also be shown that when $n$ is small, then this never occurs \cite{ivanyos2013improved}.
\end{remark}

The key technical result of \cite{ivanyos2012splitting} is that the search step can be bounded by a number which only depends on $n$ in the rational case. When $K$ is a number field then a similar algorithm can be used (with some extra technical lemmas, accounting for the fact that not all maximal orders are conjugates when the class number of $K$ is greater than 1). In that case the bound in the search step also depends on the degree and the discriminant of $K$. Further more, the bound is exponential in all the parameters ($n$, the degree and the discriminant of $K$). This implies that the algorithm is not a polynomial-time algorithm even in the case when $n=2$ and $K$ is a quadratic number field. 

In \cite{kutas2019splitting} a polynomial-time algorithm (modulo factoring integers) is proposed for the $n=2$ case when $K$ is a quadratic field. The key idea is to find a subalgebra in $A$ which is a quaternion algebra $B$ over $\mathbb{Q}$. Finding $B$ boils down to finding nontrivial solutions to quadratic forms in 3 and 6 variables. If $B$ is split, then one can find a zero divisor in $B$ efficiently. Otherwise, $B$ is a division algebra which is split by $K$. Finding a subfield isomorphic to $K$ in $B$ can be accomplished by finding a zero of a quadratic form in 4 variables. 

It is a natural question how the problem of finding an isomorphism between $A$ and $M_n(K)$ relates to finding isomorphisms between central simple $K$-algebras given by structure constants. In \cite{ivanyos2012splitting} the authors propose a method where they reduce the isomorphism problem of $A_1$ and $A_2$ to finding an isomorphism between $A_1\otimes A_2^{op}$ and $M_{n^2}(K)$. The reduction works for any computable infinite field. The idea is to find an irreducible $A_1\otimes A_2^{op}$-module $V$ which as left $A_1$-module is isomorphic to the regular representation of $A_1$ (with isomorphism $\phi_1$) and as a right $A_2^{op}$-module is isomorphic to the regular representation of $A_2^{op}$ (with isomoprhism $\phi_2$). Then one can show that $\phi_1^{-1}\circ\phi_2$ is an algebra isomorphism between $A_1$ and $A_2$.

\subsection{Function fields}

Let $K=\mathbb{F}_q(t)$ where $q$ is a prime power (which can be even in this case). First we recall the main algorithm from \cite{ivanyos2018computing} which computes an explicit isomorphism between a $M_n(\mathbb{F}_q(t))$ and an algebra $A$ given by structure constants. The key idea of the algorithm is similar to the previously described number field algorithm: find a maximal order in $A$ and try to prove that there exists a short primitive idempotent. The key observation here is that in $M_n(\mathbb{F}_q(t))$ the natural norm is non-archimedean thus matrices of norm smaller than one in any maximal order form a ring. In the number field case small elements have no structure and thus one has to do an exhaustive search to find primitive idempotents. In the function field case one can exploit this extra structure. We sketch the algorithm here: 
\begin{enumerate}
    \item Compute a maximal $\mathbb{F}_q[t]$-order $O_1$ in $A$
    \item Compute a maximal $R$-order $O_2$ in $A$ where $R$ is the subring of $\mathbb{F}_q(t)$ consisting of rational functions where the degree of the denominator is at least the degree of the numerator (i.e., the valuation ring with respect to the degree valuation)
    \item Compute the intersection $B$ of $O_1$ and $O_2$ using lattice reduction
    \item Find a complete orthogonal system of primitive idempotents in $B$, one of them will be a primitive idempotent in $A$, as well
\end{enumerate}
\begin{remark}
Elements of $O_2$ correspond to "short" elements of $A$. 
\end{remark}
In contrast to the number field case, this algorithm is polynomial in $n$ and $\log q$ due to the fact that there is no exhaustive search step at the end. This algorithm can also be used to find explicit isomorphisms between central simple $\mathbb{F}_q(t)$-algebras due to the observation described in the previous subsection. 
When $K$ is a finite extension of $\mathbb{F}_q(t)$, then the only known case is the case of separable quadratic extensions. When $q$ is odd, then \cite{ivanyos2019explicit} proposes a polynomial-time algorithm for finding zero divisors in split quaternion algebras over $K$ using a similar technique to the ones developed in \cite{kutas2019splitting}. When $q$ is even, then an analogous polynomial-time algorithm is presented in \cite{csahok2022finding}.

In \cite{gomez2020primitive} the problem of finding primitive idempotents in $A$ isomorphic to $M_n(D)$ is studied, where $D$ is a division algebra over $\mathbb{F}_q(t)$. This does not follow immediately from the previously described algorithm. The main observation is that it is enough to construct the division algebra $D$ Brauer equivalent to $A$ as then a primitive idempotent can be constructed easily. Constructing $D$ is accomplished in the following fashion: one computes the Hasse invariants of $A$ and then one constructs a division algebra with those exact Hasse invariants. When $D$ is a quaternion algebra, then this is equivalent to constructing a quaternion algebra that ramifies at specific places. In \cite{gomez2020primitive} there is a polynomial-time algorithm for constructing quaternion algebras with prescribed ramification whenever $q$ is odd. The case where $q$ is even is handled in \cite{csahok2022finding}.  


So far it is not clear, why these algorithms fail for arbitrary function fields. The reason it does not work in general is that $B$ which is the intersection of two maximal orders might just be a one-dimensional $\mathbb{F}_q$-vector space without any zero divisors again due to the fact the class number might be larger than 1.

We emphasize that some of the previously mentioned algorithms (e.g., the main algorithm from \cite{ivanyos2018computing}) have not been implemented and have no precise complexity estimate (beyond running in polynomial time). In this work we provide an implementation of \cite{ivanyos2018computing} and analyze the complexity of certain subroutines (such as maximal order computation) in more detail. 
\section{The descent method}

Let $K$ be a field and let $L$ be a separable quadratic extension of $K$. Let $A$ be a central simple algebra over $L$ given by structure constants. Our goal in this section is to find a subalgebra of $A$ which is a central simple algebra over $K$. In other words, we would like to decompose $A$ as a tensor product $B\otimes_{K}L$ when this is possible. Our main technical tool is an algorithm that computes the corestriction of a central simple algebra. We apply this in section \ref{galoisdescent} to explicit Galois descent in case of quadratic extensions.

\subsection{Explicit Galois descent}\label{galoisdescent}

Our first step is to construct an involution of the second kind on $A$ if such an involution exists. The following lemma \cite[Theorem 3.17.]{knus1998book} provides a useful relationship between certain right ideals of the corestriction of $A$ and involutions of the second kind:
\begin{lemma}\label{lem:ideal}
Let $A$ be a central simple algebra over $L$ of dimension $n^2$ where $L$ is a separable quadratic extension of the field $K$. Put $B$ for the corestriction of $A$ with respect to $L/K$. Assume that there exists a right ideal $I$ of $B$ such that  $A^{\sigma}\otimes_LA=I_L\oplus (1\otimes A)$ where $I_L=I\otimes_K L$. Then $A$ admits an involution of the second kind.
\end{lemma}
\begin{proof}
We sketch the proof here. For each $a\in A$ there exists a unique element $\tau_I(a)\in A$ such that 
$$a^{\sigma}\otimes 1-1\otimes \tau_I(a)\in I_L .$$
One can check that the map $a\mapsto \tau_I(a)$ is indeed an involution of the second kind on $A$. 
\end{proof}
Now we propose an algorithm which either returns an involution of the second kind, or a zero divisor of $A$:
\begin{algor}\label{alg:inv}
Let $L$ be a separable quadratic extension of a field $K$. Let $A$ be a central simple algebra over $L$ of dimension $n^2$ which admits an involution of the second kind.
\begin{enumerate}
    \item Compute a maximal right ideal $I$ in $B$.
    \item Let $I_L=I\otimes L$ be the scalar extension of $I$ in $A^{\sigma}\otimes A$. Compute the intersection of $I_L$ and $1\otimes A$.
    \item If $I_L \cap 1\otimes A$ is nontrivial, then we have computed a zero divisor in $A$, since every element in $I_L$ is a zero divisor.
    \item If $I_L \cap 1\otimes A$ is trivial, then $I$ is a right ideal with the property that $A^{\sigma}\otimes_LA=I_L\oplus (1 \otimes A)$ by dimension considerations which allows us to construct an involution of the second kind. 
\end{enumerate}
\end{algor}
\begin{theorem}\label{thm:inv}
Let $L$ be a separable quadratic extension of a field $K$. Let $A$ be a central simple algebra over $L$ of dimension $n^2$ which admits an involution of the second kind. Suppose that one is allowed to call an oracle for computing maximal right ideals in algebras given by structure constants which are isomorphic to $M_{n^2}(K)$ (the cost of the call is the size of the input). Then Algorithm \ref{alg:inv} runs in polynomial time. 
\end{theorem}

\begin{proof}
Let $B$ be the corestriction of $A$. Our assumptions together with Theorem \ref{involutioncorestrictionsplit} imply that $B$ is split. Thus the correctness of Algorithm \ref{alg:inv} follows from Lemma \ref{lem:ideal}. 

Now we discuss the complexity of the steps of the algorithm. Computing a right ideal is a subroutine required by the statement of the Theorem, thus Step 1 can be carried out in polynomial time. Step 2 computes the intersection of two $L$-subspaces which can be accomplished by solving a system of linear equations over $L$. Finally, the last step runs in polynomial time by Lemma \ref{lem:ideal}.  
\end{proof}

The above proof is particularly interesting when one is looking for zero divisors in quaternion algebras. 

\begin{proposition}\label{prop:subalg}
Let $L$ be a separable quadratic extension of $K$ and suppose we know an algorithm for finding explicit isomorphisms between degree 4 split central simple algebras given by structure constants and $M_4(K)$. Let $A$ be a quaternion algebra over $L$. Then one can find a quaternion subalgebra of $A$ over $K$ in polynomial time. 
\end{proposition}
\begin{proof}
Algorithm \ref{alg:inv} returns either a zero divisor or an involution of the second kind on $A$. If it returns a zero divisor, then one can efficiently construct an explicit isomorphism between $A$ and $M_2(L)$ which provides a subalgebra isomorphic to $M_2(K)$. If Algorithm \ref{alg:inv} returns an involution of the second kind, then one can compose that with the canonical involution (conjugation) on $A$. Then the fixed points of this map form a quaternion subalgebra over $K$. 
\end{proof}

When $L$ is a quadratic extension of $K=\mathbb{Q}$ or $K=\mathbb{F}_q(t)$, then there already existed efficient algorithms for computing quaternion subalgebras over $K$ in quaternion algebras over $L$ (\cite[Corollary 19]{kutas2019splitting}, \cite[Proposition 42]{ivanyos2019explicit}) using explicit calculations and utilizing algorithms for finding nontrivial zeros of quadratic form. Proposition \ref{prop:subalg} shows a more conceptual method for computing subalgebras which avoids tedious calculations. Furthermore, this proposition applies to quaternion algebras in characteristic 2 as well. 
\begin{corollary}\label{cor:char2}
Let $L$ be a separable quadratic extension of $K=\mathbb{F}_{2^k}(t)$ and $A$ be a quaternion algebra over $L$. There exists a polynomial-time algorithm which computes a quaternion subalgebra over $K$ of $A$ if such a quaternion algebra exists. 
\end{corollary}
\begin{proof}
The statement follows from Proposition \ref{prop:subalg} and the fact that there exists a polynomial-time algorithm for finding explicit isomorphisms between an algebra $A$ given by structure constants and $M_4(\mathbb{F}_{2^k}(t))$ \cite{ivanyos2018computing}.
\end{proof}
Let $L$ be a quadratic extension of $K=\mathbb{F}_{2^k}(t)$ and $A$ be an algebra isomorphic to $M_2(L)$ given by structure constants. Combining Corollary \ref{cor:char2} with \cite[Theorem 3.19]{csahok2022finding} one has the following result:
\begin{theorem}\label{thm:char2zerodiv}
Let $L$ be a quadratic extension of $K=\mathbb{F}_{2^k}(t)$ and $A$ be an algebra isomorphic to $M_2(L)$ given by structure constants. Then there exists a polynomial-time algorithm that computes a zero divisor in $A$. 
\end{theorem}
\section{The main algorithm}\label{sect:algo}

In this section we propose our main algorithm for computing explicit isomorphisms between quaternion algebras over quadratic global fields. 

We start with a small observation regarding the isomorphism problem of rational quaternion algebras. It is known that there is a polynomial-time algorithm for this task if one is allowed to call an oracle for factoring integers. Furthermore, there is a polynomial-time reduction from the problem of computing explicit isomorphisms of rational quaternion algebras to factoring, which implies that the factoring oracle is indeed necessary. 

Let $B_{p,\infty}$ be the rational quaternion algebra which is ramified at $p$ and at infinity. In \cite{eisentrager2018supersingular} the authors study the following problem: if we are given two quaternion algebras isomorphic to $B_{p,\infty}$ and we are also given a maximal order in both quaternion algebras, can we compute an explicit isomorphism between them without relying on a factoring oracle. The motivation for this problem comes from the fact that the endomorphism ring of a supersingular elliptic curve is a maximal order in $B_{p,\infty}$. The authors propose a heuristic algorithm which does not rely on factoring. Here we propose an algorithm for this task which does not rely on any heuristics: 

\begin{proposition}
Let $A,B$ be quaternion algebras isomorphic to $B_{p,\infty}$ and let $O_1,O_2$ be maximal orders in $A$ and $B$ respectively. Suppose that $A$ and $B$ are isomorphic. Then there exists a polynomial-time algorithm which computes an isomorphism between $A$ and $B$. 
\end{proposition}
\begin{proof}
In \cite{ivanyos2012splitting} the authors show that finding an isomorphism between $A$ and $B$ can be reduced to finding a primitive idempotent in $C=A\otimes_{\mathbb{Q}}B^{op}$. First observe that $O_1\otimes O_2^{op}$ is an order in $C$ which is locally maximal at every prime except at $p$. Thus we can find a maximal order containing $O_1\otimes O_2^{op}$ in polynomial time without factoring using the algorithm from \cite{voight2013identifying} (in the general algorithm one needs to factor the discriminant of the order but in this case the factorization is already known). Now we could use the algorithm from \cite{ivanyos2012splitting} but then it might only find a zero divisor which is not enough for our purposes (as it reduces to finding a zero divisor in a quaternion algebra where we do not have a maximal order). Instead we use the algorithm from \cite{ivanyos2013improved} which finds a primitive idempotent directly. 
\end{proof}
\begin{remark}
The same reasoning applies to the case where $A$ and $B$ are isomorphic rational quaternion algebras and one knows the places at which the algebras ramify. 
\end{remark}

The main goal of the remainder of the section is to design an efficient algorithm which computes an explicit isomorphism between isomorphic quaternion algebras over quadratic extensions $L$ of $\mathbb{Q}$ or $\mathbb{F}_q(t)$ (where $q$ is a prime power and can be even). In \cite[Section 4]{ivanyos2012splitting} the authors show the following reduction: 
\begin{theorem}
Let $A_1$ and $A_2$ be isomorphic central simple algebras of degree $n$ over an infinite field $K$. Then there is a polynomial-time reduction from computing an explicit isomorphism between $A_1$ and $A_2$ to computing an explicit isomorphism between $A_1\otimes A_2^{op}$ and $M_{n^2}(K)$.
\end{theorem}

Thus if one is given $A_1$ and $A_2$ which are quaternion algebras over $L$ which is a separable quadratic extension of either $K=\mathbb{Q}$ or $K=\mathbb{F}_q(t)$, then it is enough to find an explicit isomorphism between $A_1\otimes A_2^{op}$ and $M_4(L)$. Note that when $K=\mathbb{Q}$ the paper \cite{ivanyos2012splitting} proposes such an algorithm but it is exponential in the size of the discriminant of $L/\mathbb{Q}$. We will get around this issue by exploiting the fact that in this case $M_4(L)$ is not given by a usual structure constant representation but as a tensor product of two quaternion algebras. 

First we identify three algorithmic problems on which the main algorithm will rely:

\begin{problem}\label{prob:maxright}
Let $K$ be a field and let $A$ be an algebra over $K$ isomorphic to $M_4(K)$ or $M_{16}(K)$ given by structure constants. Compute a maximal right ideal of $K$. 
\end{problem}
\begin{remark}
Problem \ref{prob:maxright} is equivalent to finding an explicit isomorphism between $A$ and $M_4(K)$ or $M_{16}(K)$.
\end{remark}
\begin{problem}\label{prob:quatdiv}
Let $K$ be a field and let $D$ be a quaternion division algebra over $K$. Let $A$ be an algebra over $K$ isomorphic to $M_2(D)$ given by structure constants. Compute a zero divisor in $A$. 
\end{problem}
\begin{problem}\label{prob:quatzer}
Let $K$ be a field and let $L$ be a separable quadratic extension of $K$. Let $A$ be a split quaternion algebra over $L$ given by structure constants. Compute a zero divisor in $A$.
\end{problem}


Let $K$ be a field and let $L$ be a separable quadratic extension of $K$. We show that if one can find efficient algorithms for these problems then there exists an efficient algorithm for computing explicit isomorphisms between quaternion algebras over $L$. 

\begin{remark}
In our applications $K$ will be either $\mathbb{Q}$ or $\mathbb{F}_q(t)$. This brings up the question of why don't we just state two specific algorithms tuned to either the rational or the function field case. The reason is twofold. First, both algorithms would follow the exact same outline, only the subroutine for the aforementioned Problem \ref{prob:maxright}, \ref{prob:quatdiv},\ref{prob:quatzer} would be different. Second, if someone studied the isomorphism problem of quaternion algebras for other fields, a general framework might come in handy. More concretely, if one wanted to extend to the case where $L$ is a separable quadratic extensions of a separable quadratic extension of $\mathbb{Q}$ or $\mathbb{F}_q(t)$ then it is enough to find efficient algorithms for Problems \ref{prob:maxright}, \ref{prob:quatdiv},\ref{prob:quatzer} for the case where $K$ is a separable quadratic extension of $\mathbb{Q}$ or $\mathbb{F}_q(t)$. For example when $K=\mathbb{Q}(\sqrt{2})$, then Problem \ref{prob:maxright} admits a polynomial-time algorithm, thus only the other two have to be dealt with.     
\end{remark}

\begin{theorem}\label{thm:mainalgo}
Let $A_1$ and $A_2$ be isomorphic quaternion algebras over $L$ where $L$ is a quadratic extension of $K$. Suppose there exist polynomial-time algorithms (in the rational case polynomial-time algorithm with an oracle for factoring integers) for Problems \ref{prob:maxright}, \ref{prob:quatdiv},\ref{prob:quatzer}. Then there exists a polynomial-time algorithm for computing an isomorphism between $A_1$ and $A_2$. 
\end{theorem}

\begin{proof}
We provide an algorithm for computing an explicit isomorphism between $A_1^{op}\otimes A_2$ and $M_4(L)$. Then \cite[Section 4]{ivanyos2012splitting} implies that one can compute an explicit isomorphism between $A_1$ and $A_2$ in polynomial time.

Let $B=A_1^{op}\otimes A_2$. Then one can compute an involution of the first kind on $B$ since it is given as a tensor product of quaternion algebras (i.e., we take the ``product'' of the canonical involutions). 

Applying Theorem \ref{thm:inv} one can either construct an involution of the second kind or a zero divisor in $B$ using an efficient algorithm for Problem \ref{prob:maxright}. Suppose first that the algorithm from Theorem \ref{thm:inv} finds a zero divisor $a$ in $B$. If the zero divisor has rank 1 or 3, then one can find either a rank 1 or a rank 3 idempotent by computing the left unit of the right ideal generated by $a$. Observe that if an idempotent $e$ has rank 3, then $1-e$ has rank 1, thus one has actually found a primitive idempotent in both cases which implies an explicit isomorphism between $B$ and $M_4(L)$. If $a$ has rank 2, then we construct an idempotent $e$ of rank 2 in a similar fashion. Then $eBe\cong M_2(L)$ and computing an explicit isomorphism between them can be used to construct an explicit isomorphism between $B$ and $M_4(L)$ (as a rank one element in $eBe\cong M_2(L)$ has rank 1 in $B$). Computing an explicit isomorphism between $eBe$ and $M_2(L)$ is exactly Problem \ref{prob:quatzer}. Note that the discussion also implies that it is enough to find a zero divisor in $B$ as it can be used for constructing an explicit isomorphism between $B$ and $M_4(L)$.

Now we can suppose that the algorithm from Theorem \ref{thm:inv} has computed an involution of the second kind on $B$. We then have an involution of the second kind and an involution of the first kind on $A$. Composing them and taking fixed points finds a subalgebra $C$ of $B$ which is a central simple algebra of degree 4 over $K$ and $C\otimes_K L=B$. There are 3 kinds of central simple algebras of degree 4: full matrix algebras, division algebras, and $2\times 2$ matrix algebras over a division quaternion algebra. When $C$ is a full matrix algebra over $K$, then one can use an algorithm for Problem \ref{prob:maxright} to compute a zero divisor. When $C$ is a $2\times 2$ matrix algebra over a division quaternion algebra, then computing a zero divisor in $C$ is an instance of Problem \ref{prob:quatdiv}.  Finally, $C$ is never a division algebra as it is split by a quadratic extension (the smallest splitting field of a degree $4$ central simple algebra has degree 4 over the ground field for global fields). 
\end{proof}

After obtaining a general algorithm our goal is to look at the Problems \ref{prob:maxright}, \ref{prob:quatdiv}, \ref{prob:quatzer} in the cases where $K=\mathbb{Q}$ or $K=\mathbb{F}_q(t)$.  

\subsection{Rational function fields}

 We begin with the case when $K=\mathbb{F}_q(t)$ and $q$ is odd:
\begin{enumerate}
    \item Problem \ref{prob:maxright} can be solved in polynomial time using the main algorithm from \cite[Section 4]{ivanyos2018computing}.
    \item Problem \ref{prob:quatdiv} can be obtained in polynomial time using the algorithm from \cite[Corollary 17]{gomez2020primitive}
    \item Problem \ref{prob:quatzer} admits a polynomial-time algorithm derived in \cite[Proposition 43]{ivanyos2019explicit}.
\end{enumerate}

Now we look at the case where $q$ is even :
\begin{enumerate}
    \item Problem \ref{prob:maxright} can be accomplished in polynomial time using the main algorithm from \cite[Section 4]{ivanyos2018computing}.
    \item Problem \ref{prob:quatdiv} admits a polynomial-time algorithm by \cite[Corollary 3.22.]{csahok2022finding}
    \item Problem \ref{prob:quatzer} admits a polynomial-time algorithm by Theorem \ref{thm:char2zerodiv}
\end{enumerate}

All these imply the following:
\begin{corollary}
Let $L$ be a separable quadratic extension of $\mathbb{F}_q(t)$ where $q$ is a prime power (which can be even). Let $A_1$ and $A_2$ be two isomorphic quaternion algebras over $L$. Then there exists a randomized polynomial-time algorithm which computes an isomorphism between $A_1$ and $A_2$. 
\end{corollary}

\subsection{The rationals}

Now we turn our attention to the $K=\mathbb{Q}$ case. Problem \ref{prob:maxright} can again be accomplished in polynomial time (with the help of an oracle for factoring integers) using the algorithm from \cite[Section 2]{ivanyos2012splitting}. Problem \ref{prob:quatzer} can also be obtained in polynomial time using an oracle for factoring integers. One has to use the algorithm \cite[Corollary 19]{kutas2019splitting}. 

There is no known algorithm for Problem \ref{prob:quatdiv} in the rational case. In the rest of this section we propose a polynomial-time algorithm for this task which is analogous to \cite[Corollary 17]{gomez2020primitive}. The key ingredient of the algorithm is a result by Schwinning \cite{schwinning2011ein} (which is referred to and generalized in \cite{bockle2016division}):
\begin{theorem}
Suppose one is given a list of places $v_1,\dots,v_k$ where $k$ is even. Then there exists a polynomial-time algorithm which constructs a quaternion algebra which ramifies at exactly those places. 
\end{theorem}
\begin{proposition}
Let $A$ be an algebra isomorphic to $M_2(D)$ where $D$ is a division quaternion algebra. Then there exist a polynomial-time algorithm which is allowed to call an oracle for factoring integers which computes a zero divisor in $A$. 
\end{proposition}
\begin{proof}
First we compute a maximal order in $A$ using the algorithm from \cite[Corollary 6.5.4]{ivanyos1993finding}. An extension of this algorithm \cite{ivanyos1996algorithms} computes the places where the algebra $A$ ramifies. Now we use Schwinning's algorithm to compute a division algebra $D_0$ which ramifies at exactly those places as $A$ which implies that $A\equiv M_2(D_0)$. Now we proceed in a similar fashion as in \cite[Theorem 16 ]{gomez2020primitive} or \cite[Corollary 3.22.]{csahok2022finding} but invoking the algorithm from \cite{ivanyos2012splitting} for computing the required explicit isomorphism. 
\end{proof}

An immediate corollary is the following:
\begin{corollary}
Let $L$ be a quadratic extension of $\mathbb{Q}$ and let $A_1$ and $A_2$ be isomorphic quaternion algebras over $L$. Then there exists a polynomial-time algorithm which is allowed to call an oracle for factoring integers, that computes an explicit isomorphism between $A_1$ and $A_2$. 
\end{corollary}

\section{Complexity questions and optimisations}\label{sect:complexity}

In this section, we give complexity estimates for the computation of maximal orders in separable algebras over function fields. We then present optimisations that are relevant to our use case. More precisely, we compute maximal orders for the smallest possible algebras and use them to construct orders with small discriminant in the algebras that we generate throughout execution of algorithm \ref{algo:main}.

\subsection{Complexity of maximal order computation}
	The complexity bottleneck of our algorithm is the computation of diverse maximal orders. Although polynomial-time algorithms exist for this task (see \cite{friedrichs2000maximalorder} and \cite{ivanyos2018computing}), the actual complexity makes them rather impractical as soon as the degree of $A$ increases. Throughout the execution of algorithm \ref{algo:main}, we may encounter two $K$-algebras of degree $16$. One is the corestriction of $A = B_1 \otimes B_2$ and the other is $A_K \otimes M_2(D)$, which is done when $A_K$ itself is isomorphic to some $M_2(D)$, with $D$ a division quaternion algebra (see sections \ref{sect:algo} and \ref{sect:impl} for more details). In both cases, we need to compute a zero divisor and therefore we need to compute maximal orders (In fact, we compute a maximal order over the ring $\mathbb{F}_q[t]$ and another one over the valuation ring corresponding to the degree valuation). In the following remark, we review descriptions of the algorithm used for maximal order computations in Magma, and give an upper bound for its complexity.

	The algorithm used for computing maximal orders over Dedekind domains in associative algebras over global function fields is the one given in section 3 and 4 of \cite{friedrichs2000maximalorder}, which is similar to the algorithm described in section 3 of \cite{ivanyos2018computing}. The computation proceeds from a starting order $\Lambda_0$. Letting $\mu$ be the degree of the discriminant of $\Lambda_0$, the algorithm has a worst-case complexity of $O(\mu n^5)$, where $n$ is the dimension of the input algebra (see \cite[proposition 3.17 and remark 4.18]{friedrichs2000maximalorder}).
	If no starting order is given, one is computed from the given basis of the input algebra. However, according to the discussion in subsection 3.3 of \cite{ivanyos2018computing}, an upper bound for $\mu$ is then $2(n^8d_D + n^2d_N)$, with $d_D$ and $d_N$, where $d_D$ and $d_N$ are upper bounds respectively of the degrees of the denominators and of the numerators of the structure constants of $A$. Note that in \cite{ivanyos2018computing} $n$, is the degree of the algebra, while the convention used in \cite{friedrichs2000maximalorder} is that $n$ is the dimension. We obtain the following:
	
	\begin{proposition}\label{prop:complexity}
	    The cost of computing a maximal order in a separable $\Fq(T)$-algebra of dimension $n$ is $O(n^9)$ when the degrees of the numerators and denominators of the structure constants of $A$ are bounded.
	\end{proposition}
	
	\begin{remark}
	    \cite{ivanyos2018computing} states its result for algebras that are isomorphic to matrix algebras, but this hypothesis is not used in the estimation of bounds for the degree of the discriminant. The estimates are therefore valid for more general separable algebras.
	\end{remark}
	
	\subsection{Optimisation of the maximal order computations}\label{subsect:opti}
	
		As suggested by proposition \ref{prop:complexity}, computing maximal orders in degree $16$ matrix algebras is the computational bottleneck of our algorithm. However, this complexity depends on the degree of the discriminant of the order we start our computation with. We use this to our advantage, by computing maximal orders for the input quaternion algebras, and then passing their bases through the various operations we execute on the algebras (tensor product, corestriction and Galois descent). While it is not true that after applying these operations we always get maximal orders, we may control the growth of the discriminant, and therefore the complexity of the later maximal order computations.

	We now give results concerning the discriminant of orders passing through our various operations. In this context, $R$ is a Dedekind domain, and $K$ is the fraction field of $R$. We stress that the results given here are targeted for function fields of odd characteristic, as this is the use case of our implementation.

	\begin{proposition} \label{prop:tensorproductdisc}
		Let $A$ and $B$ be central simple algebras over $K$, respectively of dimension $m$ and $n$, and let $O_A$ and $O_B$ be $R$-orders respectively of $A$ and $B$. Then $O_A \otimes_R O_B$ is an $R$-order in $A \otimes_K B$, and $$\mathrm{Disc}(O_A \otimes_R O_B) = \mathrm{Disc}(O_A)^{n} \mathrm{Disc}(O_B)^{m}.$$
	\end{proposition}
	
	\begin{proof}
		We first note that in general, if $O$ is a $R$-algebra, the global discriminant is a product of the local ones: $\mathrm{Disc}(O) = \bigcap_{\mathfrak{p} \in \mathrm{Spec}(R)} \mathrm{Disc}(O_\mathfrak{p})$. It follows that we may localise and assume that $R$ is a PID. In particular, $O_A$ and $O_B$ are free $R$-modules.

		Let $(a_1,...,a_{m})$ be a $R$-basis of $O_A$, and $(b_1,...,b_{n})$ be a $R$-basis of $O_B$. Then since $O_a$ and $O_b$ are free $R$-modules, $(a_i \otimes b_j)_{(i,j)}$ is a $R$-basis of $O_A \otimes O_B$, and
		$$\mathrm{Disc}(O_1 \otimes O_2) = \det((\mathrm{tr}((a_{i_1} \otimes b_{j_1})(a_{i_2} \otimes b_{j_2})))_{(i_1,j_1),(i_2,j_2)})$$
		where we mean that the matrix in the determinant has its columns indexed by the couples $(i_2,j_2)$ with $1 \leq i_2 \leq m$ and $1 \leq j_2 \leq n$. Likewise, its rows are indexed by the couples $(i_1,j_1)$ with $1 \leq i_1 \leq m$ and $1 \leq j_1 \leq n$.

		We may compute reduced traces over a common splitting field for $A$ and $B$, and therefore if $a \in A$ and $b \in B$, $a \otimes b$ is a Kronecker product of matrices. It follows that $\mathrm{tr}(a \otimes b) = \mathrm{tr}(a) \mathrm{tr}(b)$. Now,
		$$\mathrm{Disc}(O_1 \otimes O_2) = \det((\mathrm{tr}(a_{i_1}a_{i_2}) \mathrm{tr}(b_{j_1}b_{j_2}))_{(i_1,j_1),(i_2,j_2)})$$
		We recognize that the matrix in the determinant is in fact the Kronecker product of matrices $(\mathrm{tr}(a_{i_1}a_{i_2}))_{1 \leq i_1 \leq m,1 \leq i_2 \leq m}$ and $(\mathrm{tr}(b_{j_1}b_{j_2}))_{1 \leq j_1 \leq n,1 \leq j_2 \leq n}$, and the lemma follows.
	\end{proof}

	Next, we consider the computation of the corestriction of a matrix algebra on a quadratic extension $K$ of a rational function field $\mathbb{F}_q(t)$ in odd characteristic, and let $\sigma$ be the non-trivial $\mathbb{F}_q(t)$-automorphism of $K$. We let $R \subsetneq \mathbb{F}_q(t)$ be a Dedekind domain, and we call $S$ the integral closure of $R$ in $K$. Let $O$ be a maximal $S$-order in $A$. Then $O \otimes_R O^\sigma$ embeds in $A \otimes_R A^\sigma$ in an obvious manner and is stable under the switch map (see definition \ref{def:switch}). We call $\Cor(O) = (O \otimes_R O^\sigma) \cap \Cor(A)$ the corestriction of $O$. We may easily construct a basis of $\Cor(O)$ in $\Cor(A)$ from a basis of $O$ in $A$. Unfortunately, $\Cor(O)$ is not a maximal $R$-order in $\Cor(A)$. However, we compute its discriminant, whose degree only depends on the quadratic field $K$. We first need a lemma:
	
	\begin{lemma}\label{lemma:uniformizer}
	With notations as above, let us assume further that $R$ is a DVR, and that its corresponding valuation in $\mathbb{F}_q(T)$ ramifies in $K$. Then $S$ admits a uniformizer $\pi$ such that $\sigma(\pi) = -\pi$.
	\end{lemma}
	
	\begin{proof}
	    Since $q$ is odd, we may find $\theta \in K \setminus \Fq(T)$ such that $\theta^2 \in \Fq(T)$. That is, $\sigma(\theta) = -\theta$. Up to multiplication by an element of $\Fq(T)$, we may assume that $\theta \in S$ and that its valuation is $0$ or $1$. Let $k$ be the residue field of $S$, then $\sigma$ induces the identity on $k$. In $k$, we therefore have $\overline{\sigma(\theta)} = \overline{\theta} = -\overline{\sigma(\theta)}$ and since $k$ has odd characteristic, $\overline{\theta} = \overline{\sigma(\theta)} = 0$. Therefore, $\theta$ is a uniformizer of $S$ and $\sigma(\theta) = -\theta$.
	\end{proof}

	\begin{proposition}\label{prop:coresram}
		Let the notations be as above. Then let $p_1,...,p_m$ be the irreducible elements of $R$ that ramify in $S$. Then $$\mathrm{Disc}(\mathrm{Cor}(O)) = \prod_{1 \leq i \leq m} p_i^{\frac{n^4-n^2}{2}}.$$
	\end{proposition}
	
	\begin{proof}
		We first prove the result in the case that $R$ is a DVR. Let $v$ be the valuation corresponding to $R$ in $K$. 

		If $v$ does not ramify in $S$, then this is proposition \ref{prop:coresunram}. We now assume that $v$ ramifies in $S$. 
		
		For the computation that follows, we will use the delta symbol for tuples. By this, we mean that if $(i,j)$ and $(o,p)$ are couples of indices, then $\delta_{(i,j),(o,p)}$ is $1$ if $(i,j) = (o,p)$ and is zero otherwise. The definition is extended to tuples with more than two elements in the obvious manner. We also will use the lexicographic order on tuples of indices.

		Let $\pi$ be a uniformizer of $S$ such that $\sigma(\pi) = - \pi$, which exists by lemma \ref{lemma:uniformizer}. Up to conjugation by an automorphism, we may assume that $O = M_n(S)$. Let $(E_{i,j})_{1 \leq i,j \leq n}$ be the canonical matrix basis of $M_n(S)$ over $S$. Then a basis of $\mathrm{Cor}(O)$ is $$B = (E_{i,j} \otimes E_{i,j})_{(1,1) \leq (i,j) \leq (n,n)}$$ $$\cup (E_{i,j} \otimes E_{k,l} + E_{k,l} \otimes E_{i,j})_{(1,1) \leq (i,j) < (k,l) \leq (n,n)}$$ $$\cup (\pi(E_{i,j} \otimes E_{k,l} - E_{k,l} \otimes E_{i,k}))_{(1,1) \leq (i,j) < (k,l) \leq (n,n)}.$$

		The discriminant of $\Cor(O)$ is then the ideal of $R$ generated by $$\det(\tr(b_ib_j))_{1 \leq i,j \leq n^4}.$$ Since $R$ is a DVR, we in fact only need to compute the valuation of this determinant in $R$.

		 We now compute the value of $\tr(b_i b_j)$ for the various choices of $b_i$ and $b_j$ in $B$. We use the general fact that $tr(E_{i,j}E_{k,l}) = \delta_{(i,j),(l,k)}$. For what follows, we consider the indices $1 \leq i,j,k,l,o,p,q,r \leq n$. We also make the assumptions that $(i,j) \neq (k,l)$ and that $(o,p) \neq (q,r)$. It is then straightforward to check the following identities. 
	\begin{align*}
		\tr((E_{i,j} \otimes E_{i,j})(E_{o,p} \otimes E_{o,p})) &= \delta_{(i,j),(p,o)} \\
		\tr((E_{i,j} \otimes E_{i,j})(E_{o,p} \otimes E_{q,r} + E_{q,r} \otimes E_{o,p})) &= 0 \\
		\tr((E_{i,j} \otimes E_{i,j})(E_{o,p} \otimes E_{q,r} - E_{q,r} \otimes E_{o,p})) &= 0 \\
		\tr((E_{i,j} \otimes E_{k,l} + E_{k,l} \otimes E_{i,j})(E_{o,p} \otimes E_{q,r} - E_{q,r} \otimes E_{o,p})) &= 0\\
		\tr((E_{i,j} \otimes E_{k,l} - E_{k,l} \otimes E_{i,j})(E_{o,p} \otimes E_{q,r} + E_{q,r} \otimes E_{o,p})) &= 0\\
		\tr((E_{i,j} \otimes E_{k,l} + E_{k,l} \otimes E_{i,j})(E_{o,p} \otimes E_{q,r} + E_{q,r} \otimes E_{o,p})) &= 2(\delta_{(i,j,k,l),(p,o,r,q)} + \delta_{(i,j,k,l),(r,q,p,o)})\\
		\tr((E_{i,j} \otimes E_{k,l} - E_{k,l} \otimes E_{i,j})(E_{o,p} \otimes E_{q,r} - E_{q,r} \otimes E_{o,p})) &= 2(\delta_{(i,j,k,l),(p,o,r,q)} - \delta_{(i,j,k,l),(r,q,p,o)})\\
	\end{align*}

	Now, the last two lines represent the trace of the product of two elements of $B$ if and only if the inequalities $(i,j) < (k,l)$ and $(o,p) < (q,r)$ are satisfied. Given $i,j,k,l$ such that $(i,j) < (k,l)$, either $(j,i) < (l,k)$ or $(l,k) < (j,i)$.

		It follows that each line of the matrix $(\tr(b_\alpha b_\beta)_{1 \leq \alpha,\beta < n^4}$, has only one non-zero coefficient. The non-zero coefficient has valuation $0$ in $S$, unless the index of the line is larger than $\frac{n^4 + n^2}{2}$, in which case the valuation is $2$. Since the matrix is symetric, this property is also true for its columns. It follows that there exists a permutation of the collumns such that the resulting matrix is diagonal. Therefore, the valuation of $\det(\tr(b_\alpha b_\beta)_{1 \leq \alpha,\beta < n^4}$ is $n^4 - n^2$ in $S$. As a result, letting $\mathfrak{p}$ be the unique maximal ideal of $R$, we get $$\mathrm{Disc}(\Cor(O)) = \mathfrak{p}^\frac{n^4-n^2}{2}.$$

		Now, let $R$ be a Dedekind domain. Then for any $R$-order $O'$, it is well known that $\Disc(O') = \bigcap_{\mathfrak{p} \in \mathrm{Spec}(R)} \Disc(O'_\mathfrak{p})$. Therefore, the result will follow from the DVR case if we prove that for $\mathfrak{p}$ a prime of $R$, $\Cor(R_\mathfrak{p} O) = R_\mathfrak{p}\Cor(O)$. However, this is immediate as multiplication by an element of $R_\mathfrak{p}$ commutes with the switch map.
	\end{proof}

	The last operation to consider is the Galois descent operation, using an involution of the second kind. It does not seem possible here to obtain such explicit results as we have had before. A reason for that is that the discriminant of the resulting $R$-order largely depends on the choice of involution of the second kind. In \cite{granath2006descent}, the situation is studied in the case of quaternion algebras.

	Following results from this subsection, we make the following optimisations to our algorithm: Maximal orders of quaternion algebras $B_1$ and $B_2$ are immediately computed. Furthermore, after applying any operation to one of our algebras, we apply the same operation to its maximal orders and then compute a maximal order of the new algebra from the order we obtain.
	
	We may now compare the efficiency of the optimised version of our algorithm and that of the naive one. The complexity estimates are given assuming that the degree of the discriminant of input quaternion algebras $B_1$ and $B_2$ is bounded. In the naive approach, we directly compute maximal orders of the corestriction of the algebra $A = B_1 \otimes B_2$. This is a call with complexity $O(n^9)$ and an input size $n = 256$. With the optimised approach, we first compute maximal orders in $B_1$ and $B_2$, which is two call with complexity in $O(n^9)$ and input size $n=8$ (recall that we count dimension over $\Fq(T)$). We must then compute a maximal order in $A = B_1 \otimes B_2$, but starting from an order with discriminant of bounded degree (see proposition \ref{prop:tensorproductdisc}). This is therefore a call with complexity in $O(n^5)$ and input size $n = 32$. Finally, we must compute a maximal order in the corestriction of $A$. This time, using proposition \ref{prop:coresram} we start from an order with discriminant $O(n)$, where $n$ is the dimension of the corestriction of $A$. This is therefore a call with complexity $O(n^6)$ and input size $n=256$. This last call is by far the most expensive of the optimised computation. We give concrete running time comparisons in subsection \ref{subsect:data}.
	
\section{Implementation}\label{sect:impl}

    In this section we present our implementation\footnote{\url{https://github.com/QuaternionIsomorphisms/QuaternionIsomorphisms/}} of algorithm \ref{algo:main} in Magma. This includes an implementation of the main algorithm from \cite{ivanyos2018computing} for computing an explicit isomorphism of a central simple algebra to a matrix algebra. This implementation, which is also used in \cite{csahok2022finding}(but in that case only on quaternion algebras), is of independent interest.
    
    We stress that due to the impracticality of algorithms for maximal order computation in algebras of dimension 256, our implementation of algorithm \ref{algo:main} currently does not terminate in reasonable time. This highlights the interest of improving the results of \cite[section 3]{ivanyos2018computing} and \cite{friedrichs2000maximalorder}, as the existence of a more efficient algorithm for this task would render our own algorithm practical. We stress that any algorithm for maximal order computation with complexity depending on the discriminant of a starting order would benefit from the optimisation described in subsection \ref{subsect:opti}.
    
    In a first subsection, we detail the subroutines we implement for \ref{algo:main}, and in a second subsection we give results of computational experiments.
    
    \subsection{Implementation details}
    For clarity of exposition, we present as algorithm \ref{algo:main} a succinct pseudo-code description of the main function in our implementation of the algorithm from theorem \ref{thm:mainalgo}. 
    
    \begin{algorithm}
	\KwIn{$(B_1,B_2)$ two quaternion algebras defined on a quadratic field $L$ over $K = \mathbb{F}_q(t)$, with $q$ odd.}
	\KwOut{A $L$-algebra isomorphism $B_1 \rightarrow B_2$.}

	$A \leftarrow B_1 \otimes_L B_2$\;
	$z,s \leftarrow \mathrm{InvolutionSecondKind}(A)$\;

	\If{z = 0}{
		$A_K \leftarrow \mathrm{Descent}(A,s)$ \;
		$z \leftarrow \mathrm{ZeroDivisor}(A_K)$\;
	}

	$e \leftarrow \mathrm{RankOneIdempotent}(A,z)$\;

	\Return{$\mathrm{IsomorphismFromIdempotent(B_1,B_2,e)}$}
	\caption{Main algorithm} \label{algo:main}
\end{algorithm}

	We now detail our implementation of the subroutines in algorithm \ref{algo:main}. In what follows, $L$ will be a quadratic extension of $\Fq(T)$.
	\begin{itemize}
		\item Tensor product computation is straightforward: one defines the algebra of dimension $16$ over $L$, with basis $(b_{1,i} \otimes b_{2,j})_{1 \leq i,j \leq 4}$. The structure constants of $A = B_1 \otimes B_2$ are then products of the structure constants of $A$ and $B$. We also construct the canonical injections from $B_1$ and $B_2$ to $B_1 \otimes B_2$. These maps are useful to give a succinct description of the conjugation involution over $B_1 \otimes B_2$ and to compute a basis of $O_1 \otimes O_2$, were $O_1$ and $O_2$ are maximal orders in $B_1$ and $B_2$.
		
		\item DescendAlgebra: Given a $L$-algebra $A$ and a semi-linear algebra automorphism $f$, we return the $K$-subalgebra of elements of $A$ fixed by $F$. We also compute low discriminant orders in this subalgebra by taking the fixed points of maximal orders of $A$ if such orders are known. The only subtlety regarding the implementation is that in order to make it efficient in Magma, the map $f$ must be defined on a $K$-vector space representing algebra $A$, since it is only semi-linear over $L$.
		
	    \item Corestriction: Computing the corestriction of an $L$-algebra $A$ is a straightforward application of proposition \ref{prop:cor}. We apply the non-trivial $\Fq(T)$-automorphism of $L$ $\sigma$ to the structure constants of $A$ to compute $A^\sigma$, and a map between $A$ and $A^\sigma$. Then maximal orders of $A$ are computed, and from them we directly obtain maximal orders of $A^\sigma$. $A \otimes A^\sigma$ and its maximal orders are computed as described above. The switch map is then computed in a straightforward manner using maps $A \rightarrow A^\sigma$, $A \rightarrow A \otimes A^\sigma$ and $A^\sigma \rightarrow A \otimes A^\sigma$. 
	    We then apply the Descent subroutine to $A \otimes A^\sigma$ and the map switch to obtain the corestriction of $A$, orders with small discriminant and a map from the corestriction to $A \otimes A^\sigma$.
	    
		\item InvolutionSecondKind: This is algorithm \ref{alg:inv}. Details of the computation of the corestriction are given below. Once the corestriction is computed, we compute a rank one idempotent $e$. Then $1-e$ generates a maximal right ideal $I$ of $B$. We therefore compute the ideal generated by $1-e$ in $A \otimes A^\sigma$. The rest is a straightforward implementation of algorithm \ref{alg:inv}.
		
		\item RankOneIdempotent when $A \simeq \mathcal{M}_n(K)$: This is the main algorithm from \cite[section 4]{ivanyos2013improved}. This algorithm uses many subroutines: we implement lattice reduction algorithms described in \cite[section 2]{ivanyos2018computing} and \cite[section 1]{lenstra1985factoring}, and the computation of the WedderburnMalcev complement of a finite algebra  following \cite[Section 3]{degraaf1997computing}. The only remaining technical part is then to compute the intersection of maximal orders in $A$ following \cite[lemma 25]{ivanyos2018computing}, and to express its structure constants as an algebra over $\mathbb{F}_q$.
		
		\item ZeroDivisor when $A \simeq \mathcal{M}_n(D)$, with $D$ a division quaternion algebra over $K$: Following \cite[Theorem 18]{gomez2020primitive}, we compute local indices of $A$ and use this information to construct a quaternion algebra $D'$ isomorphic to $D$, and then a representation of $M_m(D')$ with structure constants. We then use the RankOneIdempotent subroutine described above and the IsomorphismFromIdempotent subroutine described below to compute an isomorphism $A \simeq M_m(D')$ and return a zero divisor. Note that the hypothesis from \cite[Theorem 18]{gomez2020primitive} on the splitting places of $A$ is not needed here since we restrict to the case that $D$ is a quaternion algebra, and we therefore only need to compute local indices instead of Hasse invariants.
		
		\item RankOneIdempotent when $A \simeq \mathcal{M}_4(L)$ and a zero divisor $z$ is given: Following the discussion in the proof of theorem \ref{thm:mainalgo},  we compute $e$, the left unit of the right ideal $zA$. If $z$ has rank $1$ or $3$, we are done as per the discussion. If $z$ has rank $2$, we apply the algorithm from \cite[Proposition 43]{ivanyos2019explicit} to the split quaternion algebra $eBe$.
		
		\item IsomorphismFromIdempotent: Given a rank one idempotent in algebra $A = B_1 \otimes B_2^{op}$, we compute an explicit isomorphism $B_1 \simeq B_2^{op}$. Note that we in fact computed $A = B_1 \otimes B_2$, but since $B_2$ is a quaternion algebra, the conjugation gives an explicit isomorphism $B_2 \simeq B_2^{op}$. This is an implementation of the algorithm given by \cite[Corolary 10]{ivanyos2012splitting}.
	\end{itemize}

\subsection{Computational results}\label{subsect:data}
    In table \ref{tab:corescompare} we give running times for the task of computing maximal orders in the corestriction of a degree $2$ matrix algebra over $K=\Fq(T)(\sqrt{D})$, with $D$ a polynomial of degree $2$. The running time includes the computation of the corestriction itself. Running times are given in seconds.
    
    \begin{table}[h]
        \centering
        \begin{tabular}{cc}
            Naive version & Optimised version \\
            $95.180$ & $7.160$ \\
            $1128.870$ & $46.990$ \\
            $2338.350$ & $155.520$
         \end{tabular}
        \caption{Running time for computing maximal orders in the corestriction of degree $2$ matrix algebras}
        \label{tab:corescompare}
    \end{table}
    
    The \emph{naive version} column corresponds to the running time of the direct approach to the task. That is, computing the corestriction using linear algebra and then computing a maximal order in the corestriction algebra using the algorithm from \cite{friedrichs2000maximalorder} and \cite{ivanyos2018computing}. The worst-case complexity for this computation is $O(n^9)$, where $n = 16$ is the dimension of the corestriction algebra (see \ref{prop:complexity}). The \emph{optimised version} column shows the running time we obtain using the approach detailed in subsection \ref{subsect:opti}. The worst-case complexity then drops to $O(n^6)$. The results in table \ref{tab:corescompare} show that our optimisation is effective in practice.

	In table \ref{tab:rankone} we give running times for executions of the RankOneIdempotent subroutine from algorithm \ref{algo:main}. We execute it on a $\mathbb{F}_{17}(T)$-algebra $A$ isomorphic to $M_n(\mathbb{F}_{17}(T))$.We recall that this subroutine is an implementation of the main algorithm from \cite{ivanyos2018computing}. It begins with the computation of a maximal $\mathbb{F}_{17}[T]$-order and a maximal $R$-order of $A$, where $R$ is the valuation ring for the degree valuation. That is, $R$ is the ring of elements in $\mathbb{F}_{17}(T)$ that have a denominator of higher degree than their numerator.
	
	Running times are again given in seconds. We also give the running time of the maximal order computations.
	
	\begin{table}[h]
	    \centering
	    \begin{tabular}{cccc}
	       n  & Maximal $\mathbb{F}_{17}[T]$-order computation & Maximal $R$-order computation & Running time \\
	       $2$ &$4.690$ & $0.390$ & $5.510$ \\
	       $3$ & $7245.840$ & $401.000$ & $7706.890$
	    \end{tabular}
	    \caption{Runtime for the RankOneIdempotent subroutine}
	    \label{tab:rankone}
	\end{table}
	
	These results show that the complexity bottleneck of this subroutine is indeed the computation of maximal orders. We recall that our use case involves running this computation on algebras isomorphic to $M_{16}(\Fq)$. We conclude that our algorithm would be made practical by the discovery of a fast algorithm for computing maximal orders in separable algebras over $\Fq(T)$.

\bibliographystyle{apalike}
\bibliography{bib}

\end{document}